\documentclass[12pt]{article}
\usepackage{fullpage}
\usepackage{amsthm}
\usepackage{amsmath}
\usepackage{amssymb}
\usepackage{booktabs}
\usepackage[ruled,vlined,linesnumbered]{algorithm2e}
\usepackage{comment}
\usepackage{url}

\newtheorem{thm}{Theorem}
\newtheorem{lem}{Lemma}
\newcommand{\Z}{\mathbb{Z}}
\newcommand{\F}{\mathbb{F}}
\newcommand{\Fq}{\F_q}

\title{Primitive values of quadratic polynomials in a finite field}

\author{ Andrew R. Booker\footnote{Partially supported by EPSRC Grant EP/K034383/1.}\\
 School of Mathematics\\ University of Bristol, England\\
 andrew.booker@bristol.ac.uk
  \and
Stephen D. Cohen \\
  School of Mathematics and Statistics, \\
  University of Glasgow, Scotland \\
  stephen.cohen@glasgow.ac.uk
  \and
  Nicole Sutherland \\
  Computational Algebra Group, \\
  School of Mathematics and Statistics, \\
  University of Sydney, Australia \\
  nicole.sutherland@sydney.edu.au
\and
  Tim Trudgian\footnote{Supported by Australian Research Council Future Fellowship FT160100094.} \\
School of Physical, Environmental and Mathematical Sciences\\ The University of New South Wales Canberra, Australia \\
  t.trudgian@adfa.edu.au
}

\begin{document}

\maketitle

\begin{abstract}
  \noindent We prove that for all $q>211$, there always exists a primitive root $g$ in the finite field $\mathbb{F}_{q}$ such that $Q(g)$ is also a primitive root, where $Q(x)= ax^2 + bx + c$ is a quadratic polynomial with $a, b, c\in \mathbb{F}_{q}$ such that $b^{2} - 4ac \neq 0$. 

\end{abstract}


\section{Introduction}
For $q$ a prime power, let $\mathbb{F}_{q}$ denote the finite field of order~$q$, and let
$g_{1}$, $g_{2}$, $\ldots\,$, $g_{\phi(q-1)}$ denote the primitive roots of~$q$.
Recently \cite{COT}  the following conjecture of Cohen and Mullen \cite{CohMul} was established: that  an arbitrary  element of $\mathbb{F}_{q}$ can be written as a linear sum
of two primitive roots provided $q >61$.  In fact,  Cohen's survey of such problems \cite{Cohen1993} offered a preliminary treatment
of a further result wherein the linear sum would be  replaced by a quadratic function of primitive roots.
It is the purpose of this paper to develop this theme and provide one complete existence result on the topic.

Let
\begin{equation}\label{jelly}
a,b,c \in \mathbb{F}_{q}, \quad a\neq0,\quad b^{2} - 4ac \neq 0.
\end{equation}
Is there some $q_{0}$ such that there is always
at least one representation
\begin{equation}\label{e:vc}
g_{n} = a g_{m}^{2} + b g_{m} + c,
\end{equation}
for all $q>q_{0}$?  We  insist that $a$ be non-zero, since otherwise the result follows from the work already quoted in \cite{COT}. We also insist that $b^{2} - 4ac$ be non-zero so that $g_n$
(given by (\ref{e:vc})) is {\em not} of the form $ a(g_m+ b/(2a))^2$.

Han \cite{Han} showed under the assumption\footnote{Han and Chou et al.\
considered the additional restriction that $c$ be non-zero. We see no
reason to make this distinction.} of (\ref{jelly}), and with the
additional restriction that $q$ be odd, that one could choose $q_{0} =
2^{66}$. This was improved to $q_{0} = 2^{62}$ by Chou et al.~\cite{Chou}. Moreover, Chou et al.\ provided a list of 24 genuine exceptions to (\ref{e:vc}), all of which were not greater than 211. Upon verifying that all odd values of $q\in [223,457]$ satisfy (\ref{e:vc}), they conjectured that all $q\geq 223$ satisfy (\ref{e:vc}).

This was improved substantially by Cohen \cite[\S 3]{Cohen1993} who showed that assuming (\ref{jelly}) a representation of  the form  (\ref{e:vc})
always exists provided $q >  10^9$.  We refine this result in Theorem \ref{plate} below.

\begin{thm}\label{plate}
Let $a,b,c \in \mathbb{F}_{q}$ with $a\neq0$ and with $b^{2} - 4ac \neq 0$. Then there are primitive roots $g_{n}$ and $g_{m}$ such that
\begin{equation}\label{pumpkin}
g_{n} = a g_{m}^{2} + b g_{m} + c,
\end{equation}
for all $q$ with the exception of the values listed in $(\ref{exceptions})$.
\end{thm}
The following values of $q$ are genuine exceptions to (\ref{pumpkin}) (for some triple $(a,b,c)$).
\begin{equation}\label{exceptions}
\{2, 3, 4, 5, 9, 7, 11, 13, 16, 19, 23, 25, 29, 31, 37, 41, 43, 49, 61, 67, 71, 73, 79, 121, 127, 151, 211 \}
\end{equation}
Note that this list of exceptions agrees with that
of Chou et al.~\cite{Chou} (after noting that they only considered $q$ odd).

The outline of this paper is as follows. In \S \ref{opera} we develop the necessary character sum machinery. In \S \ref{dance} we use a sieve to prove that there are at most 1528 exceptions to Theorem \ref{plate}.  We introduce a more refined sieve in \S \ref{theatre} to reduce the number of possible exceptions to 1453. We then turn to computation in \S \ref{gallery} to complete the proof of Theorem \ref{plate}.

\section{A cast of character sums}\label{opera}
Given a positive integer $m$, let $\omega(m)$ denote the number of distinct prime factors of $m$ so that $W(m)=2^{\omega(m)}$ is the
number of square-free divisors of~$m$. Also, let $\theta(m)=\prod_{p\mid m}(1-p^{-1})$.  For any integer $m$ define its
radical $\mathrm{Rad}(m)$ as the product of all distinct prime factors\ of~$m$.

Let $e$ be a divisor of $q-1$. Call $g \in \mathbb{F}_{q}$ \emph{$e$-free} if $g \neq 0$ and
$g = h^d$, where $h \in \mathbb{F}_{q}$ and $d\mid e$, implies $d=1$. The notion of $e$-free depends
(among divisors of $q-1$) only on $\mathrm{Rad}(e)$. Moreover, in this terminology a primitive
root of $\mathbb{F}_{q}$ is a $(q-1)$-free element.

The definition of any multiplicative character $\chi$  on
$\mathbb{F}_{q}^\times$ is extended to the whole of $\mathbb{F}_{q}$ by setting $\chi(0)=0$.
For any divisor $d$ of $\phi(q-1)$, there are $\phi(d)$
characters of order
$d$, a typical character being denoted by $\chi_d$.  In particular $\chi_1$, the principal
character, takes the value $1$ at all non-zero elements of $\mathbb{F}_{q}$ (whereas $\chi_1(0)=0$).
 A convenient shorthand notation to be employed for any divisor
 $e$ of $q-1$ is
\[
  \int_{d\mid e} = \sum_{d\mid e}\frac{\mu(d)}{\phi(d)} \sum_{\chi_{d}},
\]
where the sum over $\chi_{d}$ is the sum over all $\phi(d)$ multiplicative characters $\chi_d$ of
$\mathbb{F}_{q}$ of exact order $d$.  Its significance is that, for any $g \in \mathbb{F}_{q}$,
 \[\theta(e)\int_{d\mid e}\chi_d(g) =\left\{ \begin{array}{cl}
 1, & \mbox{if }  g \mbox{ is non-zero and }  e\mbox{-free},\\
0, & \mbox{otherwise.}
\end{array} \right. \]
In this expression (and throughout) only characters $\chi_d$  with $d$ square-free contribute (even if $e$ is not square-free).

Let $e_1, e_2$ be divisors of $q-1$.  Given a triple $(a,b,c)$ satisfying (\ref{jelly}) define $N(e_1,e_2)$ to be the number of $e_1$-free elements
  $g \in  \mathbb{F}_{q}^\times$  such that
$Q(g)$ is an  $e_2$-free element in  $\mathbb{F}_{q}^\times$.
We wish to investigate when $N(q-1,q-1)$ is positive. The value of $N(e_1,e_2)$ can be expressed explicitly in terms of character sums over
$\mathbb{F}_{q}$ as follows. We have
\begin{equation}\label{char}
  N(e_1,e_2)= \theta(e_1)\theta(e_2) \int_{d_1\mid e_1} \int_{d_2\mid e_2}
    S(\chi_{d_1},\chi_{d_2}),
\end{equation}
where
\begin{equation}\label{char1}
 S(\chi_{d_1},\chi_{d_2}) = \sum_{g \in  \mathbb{F}_{q}} \chi_{d_1}(g)\chi_{d_2}(Q(g))
 \end{equation}
 with
  $Q(x)= ax^2+bx+c$. We first estimate $ S(\chi_{d_1},\chi_{d_2}) $ according as the triple $(a, b, c)$ satisfies the former or latter condition in (\ref{jelly}).
 \begin{lem} \label{Sest} Assume the triple  $(a,b,c) \in  \mathbb{F}_{q}$ satisfies $(\ref{jelly})$.
 Suppose $e_1, e_2$ are divisors  of $q-1$, $d_1\mid e_1$ and $d_2\mid e_2$.  In  $(\ref{char1})$, if $d_1=d_2=1$ then $S(\chi_1,\chi_1) \geq q-3$.
 Otherwise,
 \[  |S(\chi_{d_1},\chi_{d_2})| \leq \left\{
 \begin{array}{cl}
  2, & \mbox{ if } d_1 > 1, d_2=1,\\
 \sqrt{q} +1 , &  \mbox{ if } d_1 = 1, d_2>1,\\
 2 \sqrt{q}, & \mbox{ if } d_1 >1, d_2>1.
 \end{array}
 \right.
  \]
  \end{lem}
  \begin{proof} This follows from Weil's theorem \cite[Thm 5.41]{Lidl}, taking into account (up to) 3 zeros of $gQ(g)$.
  \end{proof}


 \begin{lem}\label{Nee} Suppose the triple $(a,b,c)$ satisfies $(\ref{jelly})$.
 Suppose that $e$ is a divisor of $q-1$ (with $e$ even if $q$ is odd).  Then
 \[ N(e,e) \geq \theta(e)^2 \left(q-2W(e)\sqrt{q}\{W(e) -3/2+1/(2W(e))\}-3W(e)  \right). \]
 \end{lem}
 \begin{proof} Write $W=W(e)$.   Applying the estimates  of Lemma \ref{Sest} to (\ref{char}) and (\ref{char1}) we obtain

 \[ N(e,e) \geq \theta(e)^2 (q-3-(W-1)^22\sqrt{q} -(W-1)(\sqrt{q}+1)- 2(W-1) )\]
and the result follows.
%
\end{proof}
We take $e=q-1$ to obtain a basic criterion to guarantee that $N(q-1,q-1)>0$.   In this situation the minor savings within Lemma \ref{Nee} are insignificant and are ignored.
 \begin{thm}\label{basic}
 Suppose the triple $(a,b,c)$ satisfies $(\ref{jelly})$. If $q > 4W(q-1)^4$
 then there exists a primitive root $g$ such that $ag^{2} + bg + c$ is also a primitive root.
 \end{thm}
  The condition in Theorem \ref{basic} is automatically satisfied if $\omega(q-1) \geq 17$.  Hence we may assume $\omega(q-1) \leq 16$
 and $q <7.37 \times 10^{19}$. To obtain an improvement on Theorem \ref{basic} we proceed, as in \cite{COT}, to introduce a sieving technique.

\section{Introducing the sieve}\label{dance}

 Let $e$ be a divisor of $q-1$. In practice, this
\emph{kernel} $e$ will be chosen such that  $\mathrm{Rad}(e)$ is the product of the smallest
primes in $q-1$.   In particular, if $q$ is odd, then certainly $e$ is even. If $\mathrm{Rad}(e) =\mathrm{Rad}(q-1)$,
then set $s=0$ and $\delta =1$.  Otherwise, if $\mathrm{Rad}(e) < \mathrm{Rad}(q-1)$
 let $p_1, \ldots, p_s$, $s \geq 1$, be the primes dividing $q-1$ but not $e$ and
 set $\delta=1-\sum_{i=1}^s 2p_i^{-1}$.
In practice, it is essential to choose $e$ so that $\delta >0$. We first borrow a result from \cite{COT}.
\begin{lem}[Lemma 1 \cite{COT}]\label{sieve} Suppose the triple $(a,b,c)$ satisfies $(\ref{jelly})$.
  Suppose $e$ is a divisor of $q-1$. Then, in the above notation,
  \begin{equation*}
    N(q-1,q-1) \geq \sum_{i=1}^sN(p_ie,e)+ \sum_{i=1}^sN(e,p_ie)-(2s-1)N(e,e).
  \end{equation*}
  Hence
  \begin{equation}\label{sieveeq2}
    N(q-1,q-1) \geq \sum_{i=1}^s\{[N(p_ie,e)-\theta(p_i)N(e,e)]+ [N(e,p_ie)-\theta(p_i)N(e,e)]\}+
    \delta N(e,e).
  \end{equation}
\end{lem}
We now proceed to use Lemma \ref{Sest} to bound the terms appearing in (\ref{sieveeq2}).
\begin{lem}\label{t:N}
Suppose the triple $(a,b,c)$ satisfies $(\ref{jelly})$.
  Let $l$ be a prime dividing  $q-1$ but not $e$. Then
  \begin{equation}\label{Neq2}
    |N(e,le) -\theta(l)N(e,e)| \leq (1-1/l)\theta(e)^{2}(2W(e)^2 \sqrt{q}- W(e)(\sqrt{q}-1)).
  \end{equation}
  and
  \begin{equation}\label{Neq3}
  |N(le,e) -\theta(l)N(e,e)| \leq (1-1/l)\theta(e)^{2}(2W(e)^2\sqrt{q}- 2W(e)(\sqrt{q}-1)).
  \end{equation}
\end{lem}
\begin{proof} From (\ref{char})
\[ N(le,e)- \theta(l)N(e,e)= \theta(le)\theta(e) \int_{d_1\mid e}
\int_{d_2\mid e}S(\chi_{ld_1},\chi_{d_2}). \]
Hence, by Lemma \ref{Sest}
 \[
\begin{array}{rl}
 |N(e,le)- \theta(l)N(e,e)| & \leq \theta(l)\theta(e)^2\left((W(le)-W(e))(W(e) -1)2 \sqrt{q} +( W(le)-W(e))(\sqrt{q}+1)\right)\\
 &=  (1-\frac{1}{l})\theta(e)^2\left(2W(e)(W(e)-1)\sqrt{q}+W(e)(\sqrt{q} +1)\right),
 \end{array} \]
 since $W(le) =2W(e)$ and $\theta(l) = 1-1/l$;  (\ref{Neq2}) follows.

 Similarly,
  \[
 |N(le,e)- \theta(l)N(e,e)|  \leq \theta(l)\theta(e)^2\left(W(e)(W(e) -1)2 \sqrt{q} + 2W(e)\right)\]
 and (\ref{Neq3}) follows.

\end{proof}

\begin{thm}\label{t:Cohen}
  Let $q$ be a prime power. Suppose the triple $(a,b,c) \in \mathbb{F}_q$  satisfies $(\ref{jelly})$. 
  Let $p_1,\ldots,p_s$, $s\geq 1$, be the primes dividing
  $q-1$ but not~$e$ and set $\delta=1-2\sum_{i=1}^s p_i^{-1}$. Suppose that $\delta$
  is positive and that
  \begin{equation}\label{e:Cohen}
    q>\left\{\left(\frac{2s-1}{\delta}+2\right)\left(2W(W-3/2)+\frac{3W}{2\sqrt{q}}\right)+1+\frac{3W}{2\sqrt{q}}\right\}^2,
  \end{equation}
  where $W=W(e)$.
  Then there is a primitive root $g$ such that $ag^{2} + bg + c$ is also primitive.

\end{thm}

\begin{proof}

  Assume $\delta>0$.  Write $N$ for $N(q-1,q-1)$ and $W$ for  $W(e)$.    From~(\ref{sieveeq2}) and  Lemmas \ref{Nee} and \ref{t:N}
  \begin{eqnarray}\label{pip}
    N & \geq & \theta(e)^2\left\{\delta\left(q-\left(2W(W-\frac{3}{2})+1+\frac{3W}{\sqrt{q}}\right)\sqrt{q}\right)-
                        \sum_{i=1}^s 2\left(1-\frac{1}{p_{i}}\right)\left(2W(W-\frac{3}{2})+\frac{3W}{2\sqrt{q}}\right)\sqrt{q}\right\} \nonumber\\
               & =    & \delta \theta(e)^2\sqrt{q}
    \left\{\sqrt{q}-\left(2W(W-\frac{3}{2})+1+\frac{3W}{\sqrt{q}}\right)-\left(\frac{2s-1}{\delta}+1\right)\left(2W(W-\frac{3}{2}) +\frac{3W}{2\sqrt{q}}\right)\right\}.
  \end{eqnarray}
  The conclusion follows.
\end{proof}
We now illustrate the utility of Theorem \ref{t:Cohen}. Recall that Theorem \ref{basic} implies that we need only consider those $q$ satisfying $\omega(q-1) \leq 16$.  A simple computation shows that, for $9 \leq \omega(q-1) \leq 16$ and $s=5$, the inequality in (\ref{e:Cohen}) is satisfied.
As an example, consider the case
$\omega(q-1)=9$, so that $q\geq 2\cdot3\cdot5\cdot7\cdot11\cdot13\cdot17\cdot19\cdot23+1=223,092,871$.
 For $s=5$ we have $W=W(e)=16$ and
$\delta\geq 1-2\bigl(\frac{1}{11}+\frac{1}{13}+\frac{1}{17}+\frac{1}{19}+\frac{1}{23}\bigr)$.
Therefore the right hand side of~(\ref{e:Cohen}) is at most $161,546,452$.  Hence there is bound to be a representation of the form
 $(\ref{pumpkin})$.

We now consider $1\leq \omega(q-1) \leq 8$ following the procedure in \S 2 of \cite{COT}.
Consider $\omega(q-1) = 8$: there is no value of $s\in[1,7]$ for which (\ref{e:Cohen}) is true. Nevertheless, we find that $s=5$ gives the smallest bound for the right-hand side of (\ref{e:Cohen}). For $s=5$ we have $W=W(e)=8$ and
$\delta\geq 1-2\bigl(\frac{1}{7}+\frac{1}{11}+\frac{1}{13}+\frac{1}{17}+\frac{1}{19}\bigr)$.
Therefore the right hand side of~(\ref{e:Cohen}) is at most $38,228,191$. We now enumerate the values of $q$ that require checking: that is, those $q$ satisfying $q\leq 38,228,191$ with $q$ a prime power and $\omega(q-1) = 8$. We find that there are 23 such values of $q$. For each of these values of $q$ we list the prime factorisation, finding the exact value of $\delta$ in each case. With this tailored approach we apply  Theorem \ref{t:Cohen} once more. We find that only 5 values of $q$ do not satisfy (\ref{e:Cohen}).

We continue in this way for\footnote{When $\omega(q-1) = 0$ we have $q=2$, whence (\ref{pumpkin}) is clearly false for $a=b=1$ and $c=0$.} $\omega(q-1) \leq 7$. In Table \ref{t:bounds} we collect our results. We run the above procedure over all values of $s\in[1, \omega(q-1)-1]$. The second column indicates the largest element in the final list. The third column contains the final number of elements of these lists, discriminating primes (on the left of the summation sign) and prime powers (on the right of the summation sign).

We have a total of 1528 possible exceptions to Theorem \ref{plate}. In the next section we introduce the modified prime sieve, which, as can be seen in the fourth column of Table \ref{t:bounds}, reduces the number of possible exceptions to 1453.

\begin{table}[ht]
  \caption{Improved bounds for $q$}
  \label{t:bounds}
  \centering
  \begin{tabular}{ccccccc}
    \hline\hline
    $\omega(q-1)$ &  Largest $q$ & Final list size (primes + prime powers) & MPS list size \\
    \hline
    $8$ &      $18888871$ &    $5+0$ & $5+0$  \\
    $7$ &    $8678671$ &   $104+1$ & $104+1$  \\
    $6$  &   $2402401$ &  $417+6$ & $403 + 6$  \\
    $5$ &    $591361$ & $477+11$ & $464 + 11$ \\
    $4$ &   $52501$ &  $378+20$ & $331 + 20$ \\
    $3$ &   $4861$ &    $73+9$ & $73 + 9$  \\
    $2$ 	& $109$   & $14 + 6$ & $14 + 5$	\\
        $1$ 	& $32$   &	$3+4$ & $3 + 4$						\\
        \textbf{Total} & & $1471 +57$ & $1397 + 56$  \\
    \hline\hline
  \end{tabular}
\end{table}

\section{The modified prime sieve}\label{theatre}
We modify the notation and argument used in Theorem $\ref{t:Cohen}$ and follow the approach of \cite{Bailey}.  Suppose that  $\mathrm{Rad}(q-1)$ is written as $ePL$, where $e$ is a divisor of $q-1$ and that for $s\geq 1$ we have
$P=p_1\cdots p_s$ is a product of distinct primes ({\em the main sieving primes}) and $L=l_1\cdots l_r (r\geq 1)$ ({\em the large primes}). In practice,
$e$ is the product of the smallest primes in $q-1$ that cannot be used as sieving primes and $L$ involves those primes that are somewhat larger than the rest (if there are any).
Write $m=\theta(e)$ and $W(e)$.  Define $\delta=1- 2\sum_{i=1}^s\frac{1}{p_i}$ and $\varepsilon= \sum_{j=1}^r \frac{1}{l_j}$.
\begin{thm} \label{mps}
Let $e\mid q-1$ as in Theorem $\ref{t:Cohen}$ and   define  $\mathrm{Rad}(q-1)=ePL, \ \delta, \varepsilon, W$, as above.     Assume $m^2\delta> 2 \varepsilon$, where
$m=\theta(e)$.  Suppose
$$\sqrt{q} > \frac{m^2(2s-1+2\delta) \left(2W(W-\frac{3}{2})+\frac{3W}{2\sqrt{q}}\right)+\delta+r-\varepsilon+\frac{1}{\sqrt{q}}\left(\frac{3m^2\delta W}{2}+2r-\varepsilon\right)}{m^2\delta-2\varepsilon}.$$
Then there is a primitive root $g$ such that $ag^{2} + bg + c$ is also primitive.
\end{thm}
\begin{proof}  Begin with the fact that $N=N(q-1,q-1)=N(ePL,ePL)$.
Then, clearly,
\begin{equation*}
 N\geq N(eP,eP)+N(L,L)- N(1,1).
\end{equation*}
Observe that, now (\ref{pip}) serves as a lower bound for the value of $N(eP,eP)$.  Moreover,
\begin{equation} \label{mps2}
N(L,L) \geq \sum_{j=1}^r[N(l_j,1)+N(1,l_j)] -(2r-1)N(1,1),
\end{equation}
because each pair $(g,Q(g)), g \in \mathbb{F}_q^\times$ contributes 1 to the right side of (\ref{mps2}) only if both $g$ and $Q(g)$ are $l_j$-free for each $j=1, \ldots, r$,
and otherwise contributes a non-positive integer.

From (\ref{mps2}), with $\Delta=N(L,L)-N(1,1)$,
\begin{equation}\label{mps3}
\Delta \geq \sum_{j=1}^r\left[\left( N(l_j,1)-\left(1-\frac{1}{l_j}\right)N(1,1)\right)+ \left( N(1,l_j)-\left(1-\frac{1}{l_j}\right)N(1,1)\right)\right]-2 \varepsilon(q-1).
\end{equation}
 From (\ref{Neq2}) and (\ref{Neq3}), for $j=1, \ldots, r$,

\[ \left|N(l_j,1) -\left(1-\frac{1}{l_j}\right)\right| \leq 2\left(1-\frac{1}{l_j}\right) \]
and
\[ \left|N(1,l_j) -\left(1-\frac{1}{l_j}\right)\right| \leq \left(1-\frac{1}{l_j}\right)(\sqrt{q}+1). \]
It follows that (\ref{mps3}) yields
\begin{equation} \label{mps4}
N(L,L)-N(1,1) \geq 2 \varepsilon (q-1) -(r-\varepsilon) \sqrt{q}-(2r-3\varepsilon).
\end{equation}
Combining (\ref{mps4}) with the relevant version of (\ref{e:Cohen}) we obtain Theorem \ref{mps}.
\end{proof}
The modified prime sieve allows us to eliminate more values of $q$ theoretically. We fed  all outstanding $q$ with $2\leq \omega(q-1) \leq 8$, as detailed in our Table \ref{t:bounds}, into the criterion of Theorem \ref{mps}. We have listed the final tally of possible exceptions in the final column of Table~\ref{t:bounds}. The use of Theorem \ref{mps} reduces the total number of possible exceptions to 1453: this list includes the following six even values of $q$: $4,8,16,32,256, 4096$. The total impact of the modified prime sieve is about a 5\% improvement on previous work. We note that in every case $r=1$ was used to eliminate a value of $q$.

It does not seem obvious how to extend these theoretical calculations further. We now turn to computational techniques to resolve Theorem \ref{plate}.

\section{Computation}\label{gallery}
\subsection{Small values of $q$}
For each possible exception $q$, we run over all admissible combinations of $a, b, c\in \mathbb{F}_{q}$ to find an appropriate primitive root. We do this in Algorithm~\ref{algo} where we consider the equivalent $a (g^2 + b g + c)$, with $b^{2} - 4c \neq 0$.

Let $R$ be the radical of $q - 1$;
then $\gamma^k$ is primitive iff $k$ is coprime to $R$.
This property is unchanged by reduction modulo $R$; hence we need
only consider $a = \gamma^k$ with $k < R$.

\begin{algorithm}[h]
\DontPrintSemicolon
\AlgoDontDisplayBlockMarkers
\SetAlgoNoEnd
\SetAlgoNoLine
\SetKwProg{Proc}{Procedure}{}{}%
\SetKwFunction{checkbetagamma}{check\_beta\_inv\_gamma}%
\SetKwFunction{checkq}{check\_q}%
\SetKwFunction{rad}{rad}%
\SetKw{KwNext}{next}%
\SetKw{KwTo}{in}%
\Proc{\checkq{$q$}}{
    \textit{Construct $\Fq$ and primitive element $\gamma$}\;
    \For{$c \in [\gamma^j : 0 \leq j < q-1] \cup [0]$}{
	\For{$b \in [\gamma^i : 0 \leq i < q-1] \cup [0]$}{
	    \If{$b^2-4 c = 0$}{
		\KwNext $b$\;
	    }
	    \For{$a \in [\gamma^k : 0 \leq k < R]$}{
		\For{$l$ in stored\_logs}{
		    \If{GCD(k+l, R) = 1}{
			\KwNext $k$\;
		    }
		}
		\For{$1 \leq m < q-1$}{
		    \If{GCD(m, R) = 1}{
			$g \leftarrow \gamma^m; \, l \leftarrow \log_{\gamma}(g^2 + b g + c)$\;
			Store $l$ in stored\_logs\;
			\If{GCD(k+l, R) = 1}{
			    \KwNext $k$\;
			}
		    }
		}
		\If{$m = q-1$}{
		    FAIL\;
		}
	    }
	}
    }
}
\caption{Check whether $q$ has the quadratic primitive property\label{algo}}
\end{algorithm}

To maximise efficiency in Algorithm~\ref{algo} we store the logarithms we
have computed as well as the elements that have already been determined to
be primitive. In this way we can first check through our list of stored primitive
elements $a$ and only generate more primitive elements as needed.

We record these results
in Table \ref{timings}. A comparison with Table \ref{t:bounds} shows that we can eliminate 
those $q$ with $1\leq \omega(q-1)\leq 3$ and over 60\% of those $q$ with $\omega(q-1) = 4$. However, as can be seen from Table \ref{timings} this approach becomes infeasible to pursue for large values of $q$. Algorithm~\ref{algo} was, however, effective in treating the six remaining even values of~$q$. 



%
\begin{table}[h]
\begin{centering}
\begin{tabular}{|c|c|c|c|c|c|c|c|c|c|}
\hline
$\omega(q-1)$ & 1 & 2 & 3 & 4 & 5 \\
Number of $q$ checked & 7 & 19 & 82 & 215 (out of 351) & 5 (out of 475) \\
Time & 1s & 3.22s & 7.8 hrs & 150 days & 4.34 days \\
\hline
\end{tabular}\caption{Total timings for checking whether $q$ has the quadratic primitive property.}\label{timings}
\end{centering}
\end{table}

\subsection{A further algorithm for odd $q$}

Assume $q$ is odd. Fix a primitive root $g$ of $\F_q^\times$, and let
$P=\{g^n:n\in\Z,\,(n,q-1)=1\}\subseteq\F_q^\times$
be the set of all primitive roots.
For a prime $p\mid q-1$, let $r_p$ denote the composition
$$
\F_q^\times\xrightarrow{\log_g}\Z/(q-1)\Z
\xrightarrow{\text{\scriptsize mod $p$}}\Z/p\Z.
$$
For $d\mid q-1$, let $H_d=\langle g^d\rangle\le\F_q^\times$
denote the unique subgroup of index $d$.

Let $d$ be an odd unitary divisor of $q-1$ --- so that $(d,(q-1)/d)=1$ --- and set
$$
e=\prod_{p\mid\frac{q-1}d}p\quad\text{and}\quad
A_d=\bigl\{g^{nd}:0\le n<e\bigr\}.
$$
Then $A_d$ is a set of coset
representatives for $H_d/H_{de}$, and it follows that
any element of $\F_q^\times$ can be expressed uniquely in the form
$\alpha hh'$, where $\alpha\in A_d$, $h\in H_{de}$
and $h'\in H_{\frac{q-1}d}$.

Let $S_d$ denote the set of prime factors of $d$, and
for each $p\in S_d$ set
$$
R_p=\left\{n+p\Z:1\le n\le\frac{p-1}2\right\}\subset(\Z/p\Z)^\times.
$$
Note that $R_p$ is a set of coset representatives for
$(\Z/p\Z)^\times/\{\pm1\}$.
For $\varepsilon=(\varepsilon_p)_{p\in S_d}\in\{\pm1\}^{S_d}$, define
$$
X_d(\varepsilon)=\bigl\{x\in\F_q^\times:r_p(x)\in\varepsilon_pR_p\cup\{0\}
\text{ for all }p\in S_d\bigr\}.
$$
Then
\begin{equation}\label{eq:cover}
\F_q^\times=\bigcup_{\varepsilon\in\{\pm1\}^{S_d}}X_d(\varepsilon).
\end{equation}

\begin{lem}\label{lem:numericaltest}
Let $d$ be an odd unitary divisor of $q-1$, and define $A_d$, $S_d$ and
$X_d(\varepsilon)$ as above.
Suppose, for all $\alpha,\beta\in A_d$ and
$\varepsilon,\delta\in\{\pm1\}^{S_d}$, that
\begin{equation}\label{eq:hyp1}
\bigl\{(\alpha x+1)^2-\beta y:x\in P\cap X_d(\varepsilon),
y\in P\cap X_d(\delta)\bigr\}\supseteq\F_q^\times
\end{equation}
and
\begin{equation*}
\bigl\{x^2-\beta y:x\in P, y\in P\cap X_d(\delta)\bigr\}
\supseteq\F_q^\times.
\end{equation*}
Then the conclusion of Theorem~$\ref{plate}$ holds for $q$.
\end{lem}
\begin{proof}
Let $a,b,c\in\F_q^\times$ with $a(b^2-4ac)\ne0$.
If $b\ne0$ then the relation $y=ax^2+bx+c$ is equivalent to
$$
\gamma=(\alpha x+1)^2-\beta y,
$$
where $\alpha=2ab^{-1}$, $\beta=4ab^{-2}$ and $\gamma=b^{-2}(b^2-4ac)$.
Similarly, for $b=0$, $y=ax^2+bx+c$ is equivalent to
$$
\gamma=x^2-\beta y,
$$
where $\beta=a^{-1}$ and $\gamma=-ca^{-1}$.
Hence, it suffices to show that
\begin{equation}\label{eq:reps1}
\bigl\{(\alpha x+1)^2-\beta y:x,y\in P\bigr\}
\supseteq\F_q^\times\quad\text{for all }\alpha,\beta\in\F_q^\times
\end{equation}
and
\begin{equation}\label{eq:reps2}
\bigl\{x^2-\beta y:x,y\in P\bigr\}\supseteq\F_q^\times
\quad\text{for all }\beta\in\F_q^\times.
\end{equation}

Consider fixed $\alpha,\beta,\gamma\in\F_q^\times$.
Let $\alpha_0,\beta_0\in A_d$,
$h_1,h_2\in H_{de}$ and $h_1',h_2'\in H_{\frac{q-1}d}$ be the elements
such that $\alpha=\alpha_0h_1^{-1}h_1'^{-1}$ and
$\beta=\beta_0h_2^{-1}h_2'^{-1}$.
By \eqref{eq:cover} we can
choose $\varepsilon,\delta\in\{\pm1\}^{S_d}$ such that
$h_1'\in X_d(\varepsilon)$ and $h_2'\in X_d(\delta)$.

Now, by hypothesis, there exist $x_0\in P\cap X_d(\varepsilon)$ and
$y_0\in P\cap X_d(\delta)$ such that
$$
(\alpha_0 x_0+1)^2-\beta_0y_0=\gamma.
$$
Hence, writing $x=h_1h_1'x_0$ and $y=h_2h_2'y_0$, we have
$\alpha x=\alpha_0x_0$ and $\beta y=\beta_0y_0$, so
that
$$
(\alpha x+1)^2-\beta y=\gamma.
$$
Further, for any $p\mid q-1$, we have $r_p(h_1)=0$, so
$$
r_p(x)=r_p(h_1')+r_p(x_0).
$$
If $p\notin S_d$ then $r_p(h_1')=0$ so that $r_p(x)=r_p(x_0)$, while
if $p\in S_d$ then
$$
r_p(x)\in(\varepsilon_pR_p\cup\{0\})+\varepsilon_pR_p
=(\Z/p\Z)^\times.
$$
In either case, we see that $r_p(x)\ne0$. Therefore $x\in P$,
and by a similar argument we find that $y\in P$.
Since $\alpha,\beta,\gamma$ were
arbitrary, we conclude that \eqref{eq:reps1} holds.

Similarly, in the $b=0$ case we choose $x\in P$ and
$y_0\in P\cap X_d(\delta)$ such that
$x^2-\beta_0 y_0=\gamma$,
and we set $y=h_2h_2'y_0$. As above we see
that $y\in P$ and $x^2-\beta y=\gamma$.
Thus \eqref{eq:reps2} holds.
\end{proof}


For given $\alpha$, $\beta$, $\varepsilon$ and $\delta$, using a fast
convolution algorithm we can compute the set \eqref{eq:hyp1} using
$O(q\log{q})$ arithmetic operations on numbers with
$O(\log{q})$ bits. Thus, the total time to check the criterion given in
Lemma~\ref{lem:numericaltest} is
$$
\ll(e2^{\omega(d)})^2q(\log{q})^{O(1)}
\ll_\varepsilon e^2q^{1+\varepsilon}.
$$
For very large $q$ we expect the criterion to be satisfied even with
$e=2$, so that Theorem~\ref{plate} can be verified in quasi-linear time
$O(q^{1+\varepsilon})$. However, that is slightly misleading, since
we will only apply the algorithm to those $q$ satisfying
$q\le4W(q-1)^4$, in which case the running time is
$$
\gg W(q-1)^2q\log{q}\gg q^{3/2}\log{q}.
$$
There is also significant overhead in the convolution algorithm, to
the point that we found the naive method of enumerating all values of
$(\alpha x+1)^2-\beta{y}$ to be faster in practice. That method requires
approximately $(e\phi(q-1))^2$ arithmetic operations in $\F_q$,
which is still reasonable for $q\le18,888,871$ on modern computers,
provided that $e$ is not too large.

Write the prime factorisation of $q-1$ as $\prod_{i=1}^sp_i^{e_i}$,
where $2=p_1<\cdots<p_s$. We coded the criterion of
Lemma~\ref{lem:numericaltest} with $d=\prod_{n<i\le s}p_i^{e_i}$ for some
$n\in\{1,\ldots,s\}$. We first try $n=1$ (corresponding to $e=2$), then
$n=2$, and so on, until either the criterion is satisfied or we reach
$n=s$ without success. Note that when $n=s$ (corresponding to $d=1$),
our algorithm becomes an exhaustive search, so it must eventually succeed
whenever the conclusion of Theorem~\ref{plate} holds for $q$.

Applying this procedure (see \cite{code}) to the 1447 odd values of $q$ that were not
covered by Theorems~\ref{t:Cohen} and \ref{mps}, we found that most
succeeded with $n=1$ or $2$. In particular, for all $q>150,151$ the
algorithm succeeded for some choice of $e\le6$.


\section*{Acknowledgements}
The second and fourth authors would like to thank Tom\'{a}s Oliveira e Silva with whom we had several discussions on preliminary versions of this paper.

\end{document}